\documentclass{birkjour}
 \newtheorem{thm}{Theorem}[section]
 
 \newtheorem{lem}[thm]{Lemma}
 \newtheorem{prop}[thm]{Proposition}
 \theoremstyle{definition}
 
 \theoremstyle{remark}
 \newtheorem{rem}[thm]{Remark}
 
 \numberwithin{equation}{section}

\newcommand{\Hol}{{\mathcal H}}
\newcommand{\D}{{\mathbb D}}
\newcommand{\C}{{\mathbb C}}
\newcommand{\w}{\omega}

\newcommand{\dom}{\mathcal D}
\newcommand{\ran}{\mathcal R}
\newcommand{\uP}{{\mathbb U}}
\newcommand{\R}{{\mathbb R}}

\newcommand{\al}{\alpha}
\newcommand{\la}{\lambda}
\newcommand{\g}{\gamma}
\newcommand{\G}{\Gamma}
\newcommand{\vj}{\varphi}
\newcommand{\s}{\sigma}

\newcommand{\Aut}{\operatorname{Aut}}

\begin{document}

%-------------------------------------------------------------------------

%---------------------------------------------------------------------------

%

\title[Duality of the non-reflexive Bergman space of the upper half plane]
 {Duality of the nonreflexive Bergman space of the upper half plane and Composition groups}

%----------Author 1
\author[E. O. Gori]{E. O. Gori}
\address{%
Department of Pure and Applied Mathematics \\
Maseno University\\
P.O. Box 333 - 40105\\
Maseno\\
Kenya}
\author[J. O. Bonyo]{J. O. Bonyo}
\address{%
Department of Pure and Applied Mathematics \\
Maseno University\\
P.O. Box 333 - 40105\\
Maseno\\
Kenya}
\email{jobbonyo@maseno.ac.ke}

%\thanks{This work was completed during the period when the second author was visiting the Aristotle University of Thessaloniki, Greece. He would like to sincerely that the Simon's Foundation for funding his visit. He would also wish to thank his host Prof. Aristomenis G. Siskakis and the department of Mathematics for the unmatched hospitality}



%\dedicatory{Dedicated to Prof. G. K. Rao on his retirement}
%%% ----------------------------------------------------------------------

\begin{abstract}
We identify the predual of the nonreflexive Bergman space of the upper half plane, $L_a^1(\uP,\mu_{\al})$, with the little Bloch space of the upper half plane consisting of functions vanishing at $i$. We then investigate both the semigroup and spectral properties of the adjoint groups of composition operators which are naturally obtained from the duality pairing and are therefore defined on the identified predual.
%Using the Cayley transform as well as related works on the unit disk by Zhu, Peloso among others, we obtain the predual of non-reflexive Bergman space of the upper half plane $L ^{1} _{a} (\mathbb{U} , \mu_{\alpha})$. We determine the weighted composition operator on the predual space of $L ^{1} _{a} (\mathbb{U} , \mu_{\alpha})$. Finally, we employ the theory of semigroups and spectral to determine infinitesimal generator of the weighted composition operator obtained, then establish the strong continuity property and obtain the resolvents of the infinitesimal generator which is given as the integral operators.
\end{abstract}

%%% ----------------------------------------------------------------------
\maketitle
%%% ----------------------------------------------------------------------
%\tableofcontents
\section{Introduction}
Let $ \mathbb{C}$ be the complex plane. The set $ \mathbb{D} := \{z \in \mathbb{C} : |z| < 1\}$ is called the open unit disc. Let $dA $ denote the area measure on $\mathbb{D}$, normalized so that the area of $\mathbb{D}$ is 1. In terms of rectangular and polar coordinates, we have:
         $dA(z)= \frac{1}{\pi} dxdy = \frac{r}{\pi} drd \theta $,
where $z = x+ iy= re^{i\theta} \in \mathbb{D}$.  For $\alpha \in \mathbb{R} , \alpha > -1 $, we define a positive Borel measure $dm _{\alpha}$ on $\mathbb{D}$ by $dm _{\alpha} (z) = (1- |z| ^{2})^{\alpha} dA(z)$, and thus $dm _{\alpha}$ is a probability measure. Moreover, if $\alpha = 0$, then   $dm _{o} = dA$. We consider $dm _{\alpha}$ as a weighted measure and a generalization of $dA$.
On the other hand, the set $\mathbb{U}: = \{\omega \in \mathbb{C} : \Im (\omega) > 0 \}$ denotes the upper half of the complex plane $\mathbb{C}$, with $\Im(\omega)$ being the imaginary part of $\omega \in \mathbb{C}$. For $\alpha > -1$, we define a weighted measure on $\mathbb{U}$ by $d\mu _{\alpha} (\omega) = (\Im(\omega)) ^{\alpha} dA(\omega)$,
where $\omega \in \mathbb{U} $. Again it can easily be seen that $\alpha = 0$ coincides with the unweighted measure. The function $\psi (z) = \frac{i (1+z)}{1-z}$ is referred to as the Cayley transform and maps the unit disc $\mathbb{D}$ conformally onto the upper half-plane $\mathbb{U}$ with the inverse $\psi^{-1} (\omega) = \frac{\omega -i}{\omega + i}$.\\
For an open subset $\Omega$ of $ \mathbb{C}$, let $\mathcal{H} (\Omega)$ denote the space of analytic functions on $\Omega$. For $1 \leq p < \infty $, $\alpha > -1$, the weighted Bergman space of the upper half-plane $\mathbb{U}$ is defined by
\[L ^{p} _{a} (\mathbb{U}, \mu _{\alpha}) : = \left\{f \in \mathcal{H} (\mathbb{U}) : \| f \| _{L ^{p} _{a} (\mathbb{U}, \mu _{\alpha})} = \left (\int_{\mathbb{U}} |f(z)| ^{p} d \mu _{\alpha}(z)\right) ^{\frac{1}{p}} < \infty \right \}.\]
In particular, $L ^{p} _{a} (\mathbb{U}, \mu _{\alpha})=L ^{p}  (\mathbb{U}, \mu _{\alpha}) \cap \mathcal{H} (\mathbb{U})$, where $L ^{p}  (\mathbb{U}, \mu _{\alpha})$ or simply $L ^{p}  ( \mu _{\alpha})$ denotes the classical Lebesque spaces with respect to the weighted measure $d\mu _{\alpha}$. It is important to note that the case $\alpha = 0 $ yields the unweighted Bergman space.
$L ^{p} _{a} (\mathbb{U}, \mu _{\alpha})$ is a Banach space with respect to the norm
\begin{eqnarray*}
 \| f \| _{L ^{p} _{a} (\mathbb{U}, \mu _{\alpha})} &=& \left(\int_{\mathbb{U}} |f(z)| ^{p} d \mu _{\alpha}(z)\right) ^{\frac{1}{p}} < \infty.
\end{eqnarray*}
For $p=2$, $L ^{2} _{a} (\mathbb{U}, \mu _{\alpha})$ is a Hilbert space.
The growth condition for the weighted Bergman space functions is given by: For every $f \in L ^{p} _{a} (\mathbb{U}, \mu _{\alpha})$, $\g=\frac{\al+2}{p}$ and $ \omega \in \mathbb{U}$, there exists a constant $K$ such that,
\begin{eqnarray*}
  |f(\omega)|&\leq&\frac{K\|f\|}{(\Im(\omega))^{\gamma}}.
\end{eqnarray*}
For a detailed account of the theory of Bergman spaces, we refer to \cite{Dur2, Pel, K.ZH}.\\
On the other hand, the Bloch space of the unit disk, denoted by $B_{\infty} (\mathbb{D})$, is defined by
\[B_{\infty} (\mathbb{D}):= \{f\in \mathcal{H}(\mathbb{D}):\| f \|_{B_{\infty,1} (\mathbb{D})}=\sup_{z\in \mathbb{D}} (1-|z|^{2})|f^{\prime} (z)|<\infty \},\]
 with the norm on $B_{\infty} (\mathbb{D})$ is given by $ \| f \|_{B_{\infty} (\mathbb{D})}:= |f(0)| + \| f \|_{B_{\infty,1} (\mathbb{D})}$, while $\|. \|_{B_{\infty,1} (\mathbb{D})}$ is a seminorm. \\The Bloch space of the upper half plane denoted by $B_{\infty} (\mathbb{U})$ is defined by
\[B_{\infty} (\mathbb{U}):= \{f\in \mathcal{H}(\mathbb{U}) :\| f \|_{B_{\infty,1} (\mathbb{U})}=\sup_{\omega\in \mathbb{U}} \Im(\omega)|f^{\prime} (\omega)|<\infty \},\]
with the norm given by $\| f \|_{B_{\infty} (\mathbb{U})}=|f(i)|+ \| f \|_{B_{\infty,1} (\mathbb{U})}$.
The little Bloch space of the unit disk denoted by $B_{\infty,\circ} (\mathbb{D})$ is defined as
\[B_{\infty,\circ} (\mathbb{D}):= \{f\in \mathcal{H}(\mathbb{D}) :\lim_{|z|\to 1} (1-|z|^{2})|f^{\prime} (z)|=0 \}\] but with the same norm as $B_{\infty} (\mathbb{D})$, while for the upper half-plane, the little Bloch space is denoted by $B_{\infty,\circ} (\mathbb{U})$ and is defined by \[B_{\infty,\circ} (\mathbb{U}):= \{f\in \mathcal{H}(\mathbb{U}) :\lim_{\Im(\omega)\to 0} \Im(\omega)|f^{\prime} (\omega)|=0 \}\] with the same norm as $B_{\infty} (\mathbb{U})$. For a comprehensive theory of Bloch spaces, see \cite{K.ZH, Zhu2}.\\
The duality properties of Bergman spaces are well known in literature. For instance in \cite[Theorem 4.2.9]{K.ZH}, it is proved that for $ 1 < p < \infty $, $\frac{1}{p} + \frac{1}{q} = 1$ and $\alpha > -1 $, the dual space of the Bergman space $L_{a}^{p}(\mathbb{D} , m_{\alpha})$ is given by
$(L_{a}^{p}(\mathbb{D} , m_{\alpha}))^{*} \approx L_{a}^{q}(\mathbb{D} , m_{\alpha})$
under the duality pairing,
\begin{eqnarray*}
\langle g,f\rangle &=& \int _{\mathbb{D}} g(z) \overline{f(z)} dm_{\alpha}\quad (g\in L_{a}^{p}(\mathbb{D} , m_{\alpha}) ,f \in L_{a}^{q}(\mathbb{D} , m_{\alpha})).
\end{eqnarray*}
For the non-reflexive Bergman space on the unit disk, $L_a^1(\D,m_{\al})$, it is shown in \cite[Theorems 5.1.4 and 5.2.8]{K.ZH} that the dual and predual spaces of $L ^{1} _{a} (\mathbb{D},m_{\alpha})$ are the Bloch space and the little Bloch space respectively. In particular,
 $ (L_{a}^{1}(\mathbb{D} , m_{\alpha}))^{*} \approx B_{\infty} (\mathbb{D})$ and $(B_{\infty,\circ} (\mathbb{D}))^{*}   \approx  L_{a}^{1}(\mathbb{D} , m_{\alpha})$ under the duality pairings given by respectively,
\begin{eqnarray*}
\langle g,f\rangle = \int _{\mathbb{D}} g(z) \overline{f(z)} dm_{\alpha}(z)\quad(g\in L_{a}^{1}(\mathbb{D} , m_{\alpha}), f \in B_{\infty} (\mathbb{D})),
\end{eqnarray*}
 and
\begin{eqnarray*}
  \langle g,f\rangle &=& \int _{\mathbb{D}} g(z) \overline{f(z)} dm_{\alpha}(z)\quad(f\in L_{a}^{1}(\mathbb{D} , m_{\alpha}), g \in B_{\infty,\circ} (\mathbb{D})).
\end{eqnarray*}
For the corresponding spaces of the upper half plane, it has been proved and noted that the dual space of the reflexive Bergman space of the upper half plane $L ^{p} _{a} (\mathbb{U},  \mu_{\alpha})$ is $L ^{q} _{a} (\mathbb{U} , \mu_{\alpha})$ for $1 < p,q < \infty$ with $\tfrac1{p}+\tfrac1{q}=1$ under a similar pairing as above. See for instance, \cite{S.A, BBMM, Bon} or \cite{Pel} for details. When $p=1$, the space $L ^{1} _{a} (\mathbb{U} , \mu_{\alpha})$ is non-reflexive, and it's recently that the dual was determined by Kang \cite{Kang} as we give in Theorem \ref{Theorem 7} stated in the next section. Apparently, the predual of $L_a^1(\uP,\mu_{\al})$ is not explicitly clear from the literature. Generally, there's no unified and comprehensive exposition of properties of the analytic spaces of upper half plane $\uP$ as there is for the corresponding spaces on the unit disk $\D$. Therefore, the first focus of this paper is to determine the predual of $L_a^1(\uP,\mu_{\al}),$ that is, identifying the space whose dual is $L_a^1(\uP,\mu_{\al})$.\\
Let $\Aut(\uP)$ denotes the collection of all automorphisms of $\uP$. For $\vj_t \in \Aut(\uP)$, $t\geq 0$, we define a composition operator on $\Hol(\uP)$ by $C_{\vj_t}f := f\circ \vj_t$. The corresponding group of weighted composition operator on $\Hol(\uP)$ is therefore given by $T_tf: = S_{\vj_t}f=(\vj_t')^{\g}f\circ \vj_t$ for some appropriate weight $\g$. Motivated by the work of Arvanitidis and Siskakis in \cite{S.A}, the current second author and three others in \cite{BBMM} classified all the self - analytic maps of the upper half plane into three distinct groups, namely: the scaling, the translation and the rotation groups. They then studied both the semigroup and spectral properties of the corresponding groups of weighted composition operators. As for the properties of the adjoint groups on the reflexive weighted Bergman spaces $L_a^p(\uP,\mu_{\al})$, $1<p<\infty,$ only the scaling group was considered in \cite{BBMM} and later completed for the other two groups by the second author in \cite{Bon}. In this paper, we investigate the adjoint properties of the groups of weighted composition operators on nonreflexive Bergman space $L_a^1(\uP,\mu_\al)$.\\
Let $X$ and $Y$ be  Banach spaces over $\mathbb{C}$. The space $\mathcal{L}(X,Y) = \{{T: \,X\,\to\, Y}$ such that $T$ is linear and continuous\}, endowed with the operator norm $\|T\| = \sup_{\|x \| \leq 1} \|Tx\|$, is a Banach space \cite{Con}. We write $\mathcal{L}(X,X) = \mathcal{L} (X)$. $T$ is said to be a closed operator if its graph $\{(x,Tx) \mid x\in \dom(T)\}$ in $X \times Y$ is closed.
 Let $T$ be a closed operator on $X$. The resolvent set of $T$, $\rho (T)$ is given by
 $\rho (T)= \{{\lambda \in \mathbb{C} : \lambda I-T }$ is invertible or bijective\} and its spectrum $\sigma (T)= \mathbb{C}\ \setminus\rho (T)$. Therefore
 $\sigma (T) \cup \rho (T) = \mathbb{C}$.
 The spectral radius of $T$ is defined by $r(T) = \sup \{|\lambda| : \lambda \in \sigma (T)\}$ with  the relation $r(T) \leq \| T \|$.
 The point spectrum $\sigma _{p}(T)= \{{ \lambda \in \mathbb{C} : Tx = \lambda x }$ for some $0 \neq x \in \mbox{dom} (T)\}$.
 For $\lambda \in \rho (T)$, the operator $R(\lambda ,T ) : = (\lambda I - T)^{-1} $ is, by the closed graph theorem a bounded operator on $X$ and is called the resolvent of $T$ at the point $\lambda$ or simply the resolvent operator. In fact, $\rho (T)$ is an open subset of $\mathbb{C}$ and $R(\lambda,T) : \rho (T) \to \mathcal{L} (X)$ is an analytic function. For a detailed theory on spectra, we refer to \cite{Con, Dun, Neu, rudin}.

\section{Predual of Non-reflexive Bergman space of the upper half-plane $L^{1}_{a}(\mathbb{U},\mu_{\alpha})$ }
Let $B_{\infty}(\mathbb{U},i)$ denote the subspace of the Bloch space $B_{\infty}(\mathbb{U})$ consisting of functions vanishing at $i$. Therefore $B_{\infty}(\mathbb{U},i)$ is defined as
 \begin{eqnarray*}
   B_{\infty}(\mathbb{U},i)&:=&\{ f\in B_{\infty}(\mathbb{U}):f(i)=0\}.
 \end{eqnarray*}
Then $B_{\infty}(\mathbb{U},i)$ is a closed subspace of $B_{\infty}(\uP)$ and therefore is a Banach space with respect to the norm $\| f\|_{B_{\infty,i}}:=\| f\|_{B_{\infty}(\mathbb{U})}=\| f\|_{B_{\infty,1}(\mathbb{U})}$.
Similarly, let $B_{\infty,\circ}(\mathbb{U},i)$ denotes the subspace of  $B_{\infty,\circ}(\mathbb{U})$ consisting of functions vanishing at $i$. Therefore
\begin{eqnarray*}
 B_{\infty,\circ}(\mathbb{U},i)&:=& \{ f\in B_{\infty,\circ}(\mathbb{U}):f(i)=0\},
\end{eqnarray*}
with the norm $\| f\|_{B_{\infty,i}}:=\| f\|_{B_{\infty}(\mathbb{U})}=\| f\|_{B_{\infty,1}(\mathbb{U})}$. Again, $B_{\infty,\circ}(\uP,i)$ is a Banach space with respect to the norm given above.\\
The following result due to Kang \cite{Kang} gives the dual of $L_a^1(\uP,\mu_\al)$;
\begin{thm}\label{Theorem 7}
For any $\alpha\in\mathbb{R}$, $\alpha > -1$, we have
\begin{eqnarray*}
  (L^{1}_{a}(\mathbb{U},\mu_{\alpha}))^{*} &\approx& B_{\infty}(\mathbb{U},i),
\end{eqnarray*}
 under the integral pairing
\begin{eqnarray*}
  \langle g,f\rangle &=& \int_{\mathbb{U}}g(w)\overline{f(w)}d\mu_{\alpha}(w)\quad(g\in L^{1}_{a}(\mathbb{U},\mu_{\alpha}),f \in B_{\infty}(\mathbb{U},i)).
\end{eqnarray*}
\end{thm}
With the help of Theorem \ref{Theorem 7} above, we determine the predual space of $L^{1}_{a}(\mathbb{U},\mu_{\alpha})$, that is, a set whose dual is $L^{1}_{a}(\mathbb{U},\mu_{\alpha})$, but first we state some results.\\
Let $\mathbb{C} (\overline{\mathbb{U}})$ be the algebra of complex valued continuous functions on $\overline{\mathbb{U}}$, and  $\mathbb{C}_{\circ} (\overline{\mathbb{U}})$ be the subalgebra of $\mathbb{C} (\overline{\mathbb{U}})$ consisting of functions $f$ such that $f(\omega) \to 0 $ as $\Im(\omega)\to 0$.
\begin{prop}\label{prop 1}
$\mathbb{C}_{\circ} (\overline{\mathbb{U}}):=\{ g\circ\psi^{-1} :g\in\mathbb{C}_{\circ} (\overline{\mathbb{D}})\}. $
\end{prop}
\begin{proof}
Let $\mathbb{K}\subset \mathbb{U}$ be compact. Since Cayley transform $\psi: \mathbb{D}\to \mathbb{U}$ is a continuous bijection, it follows that $\mathbb{K}\subset \mathbb{U}$ is compact if and only if $\psi^{-1}(\mathbb{K})$ is compact in $\mathbb{D}$. If $f\in \mathbb{C}_{\circ}(\mathbb{U})$ and $\epsilon>0$, then there exists $\mathbb{K}$ compact in $\mathbb{U}$ such that $\sup_{w\in\mathbb{U}\setminus\mathbb{K}}|f(w)|< \epsilon$.\\
Now, $g=f\circ \psi$ is continuous on $\mathbb{D}$ with $f=g\circ \psi^{-1},$ and
\begin{eqnarray*}
   \sup_{z\in\mathbb{D}\setminus\psi^{-1}(\mathbb{K})}|g(z)|&=&\sup_{w\in\mathbb{U}\setminus\mathbb{K}}|f(w)| < \epsilon.
\end{eqnarray*}
\end{proof}
\begin{prop}\label{prop1.2}
Let $C_{\psi}$ be the composition by $\psi$ operator. Then
\begin{enumerate}
  \item $f\in B_{\infty}(\uP)$ if and only if $C_{\psi}f \in B_{\infty}(\D)$. In particular, $\|f\|_{B_{\infty,1}(\uP)}=\tfrac1{2}\|C_{\psi}f\|_{B_{\infty,1}(\D)}.$
  \item $f\in B_{\infty,\circ}(\uP)$ if and only if $C_{\psi}f \in B_{\infty,\circ}(\D)$.
  \item $f\in L^1(\uP,\mu_\al)$ if and only if $S_{\psi}f \in L^1(\D,m_\al)$. In particular, $\|f\|_{L_a^1(\uP,\mu_\al)}=\tfrac1{2^{\al}}\|S_{\psi}f\|_{L^1(\D,m_\al)}$.
  \item $f\in L^{\infty}(\uP,\mu_\al)$ if and only if $C_{\psi}f \in L^{\infty}(\D, m_\al)$.
\end{enumerate}
\end{prop}
\begin{proof}
For (1), if $f \in B_{\infty} (\uP)$, then by definition,
\begin{align*}
\|f\|_{B_{\infty,1} (\uP)} &= \sup_{\w\in\uP}(\Im(\w))|f'(\w)| = \sup_{z\in\D} \frac{1-|z|^2}{|1-z|^2}\left|f'(\psi(z))\right|\\
&= \frac1{2} \sup_{z\in\D} (1-|z|^2)|\psi'(z)||f'(\psi(z))| = \frac1{2} \sup_{z\in\D} (1-|z|^2)|(f\circ\psi)'(z)|\\
&=\frac1{2} \|f\circ\psi\|_{B_{\infty,1}(\D)}.
\end{align*}
For (2), we have $f\in B_{\infty,0}(\uP)$ is equivalent to
\begin{align*}
\lim_{\Im(\w)\to 0} (\Im(\w))|f'(\w)| &= \lim_{\Im(\psi(z)) \to 0} \frac{1-|z|^2}{|1-z|^2} |f'(\psi(z))| = \frac1{2} \sup_{|z| \to 1} (1-|z|^2)|(f\circ\psi)'(z)|=0,
\end{align*}
which in turn is equivalent to $f\circ\psi \in B_{\infty,0}(\D)$, as desired. For $f\in L_a^1(\uP,\mu_\al)$, we have
\begin{align*}
\|f\|_{L_a^1(\uP,\mu_\al)} &= \int_{\uP}|f(\w)|\,d\mu_\al(\w)\\
&= \int_{\uP}|f(\w)|\Im (\w)^\al \,dA(\w)\\
&= \int_{\D}|f(\psi(z))|\left(\frac{1-|z|^2}{|1-z|^2} \right)^{\al}|\psi'(z)|^2 \,dA(z)\\
&= \tfrac1{2^{\al}} \int_{\D}|f(\psi(z))| |\psi'(z)|^{\al +2}\,dm_{\al}(z)\\
&= \tfrac1{2^{\al}} \int_{\D}(\psi'(z))^{\g}|f(\psi(z))|\,dm_{\al}(z)\\
&= \tfrac1{2^{\al}} \|S_{\psi}f\|_{L^1(\D, m_{\al})},
\end{align*}
which proves (3). Now, $f\in L^{\infty}(\uP,\mu_{\al})$ means that $f$ is essentially bounded which implies that $f\circ \psi$ is essentially bounded as well. Since $\psi$ is an invertible mapping from $\D$ onto $\uP$, it follows that $f\circ \psi \in L^{\infty}(\D,m_{\al})$. The converse follows similarly. This completes the proof.
\end{proof}
\begin{rem}
It is easy to verify that $C_{\psi^{-1}} = C_{\psi}^{-1}$. Proposition \ref{prop1.2} above therefore implies that $C_{\psi}$ is an is an isometry up to a constant and at the same time invertible on the respective spaces with the inverse also acting on the appropriate spaces.\\
More generally, let $\{V_1,V_2\}=\{\D,\uP\}$, and let
$LF(V_i,V_j)$ denote the collection of conformal mappings from $V_i$ onto $V_j$. Then $LF(V_i,V_i)=\Aut(V_i)$, and if $h\in LF(V_i,V_j)$, then $g\in\Aut(V_j)\mapsto h^{-1}\circ g\circ h\in\Aut(V_i)$ is an isomorphism from $\Aut(V_i)$ onto $\Aut(V_j)$.
For each $g\in LF(V_i,V_j)$, we define a weighted composition operator $S_g:\;\Hol(V_j)\;\to\; \Hol(V_i) $,  by
\begin{equation} S_gf(z)\;=\;(g'(z))^\g f(g(z)),\quad\mbox{ for all } z\in V_i.\end{equation}
We note that if $g\in LF(V_i,V_j)$ and $h\in LF(V_j,V_i)$, then it is clear by chain rule that $S_h S_g=S_{gh}$ and $S_g^{-1}=S_{g^{-1}}$.
\end{rem}
Using Propositions \ref{prop 1} and \ref{prop1.2} above, we obtain the following result which is the upper half-plane analogue of \cite[Lemma 5.14]{K.ZH}.
\begin{prop}\label{prop 2}
For $t>0$, $\alpha > -1$, let the integral operator $T$ on $\mathcal{H}(\mathbb{D})$ be defined by
\begin{eqnarray*}
  T f(z) &=& (1-|z|^{2})^{t}\int_{\mathbb{D}}\frac{f(w)}{(1-z\overline{w})^{2+t+\alpha}}dm_{\alpha}(w).
\end{eqnarray*}
Let $S$ be the corresponding integral operator on $ \mathcal{H}(\mathbb{U})$ defined by \\$S:= C_{\psi ^{-1}} T C_{\psi}$. Then the following properties hold:
\begin{itemize}
  \item[(a)] $S= (\alpha +t+1)S^{2}$,
  \item [(b)] $S$ is a bounded embedding of $B_{\infty}(\mathbb{U})$ into $L^{\infty}(\mathbb{U})$ and
  \item [(c)] $S$ is an embedding of $B_{\infty,\circ}(\mathbb{U})$  into $\mathbb{C}_{\circ}(\mathbb{U})$.
\end{itemize}
\end{prop}
\begin{proof} From \cite[Lemma 5.14]{K.ZH}, we have,
\begin{eqnarray*}
   S &= &C_{\psi ^{-1}} T C_{\psi}=C_{\psi ^{-1}}(\alpha+t+1) T^{2} C_{\psi}  \\
   &= &(\alpha+t+1)C_{\psi ^{-1}} T^{2} C_{\psi}  \\
   &= &(\alpha+t+1)S^{2},
\end{eqnarray*}
which proves (a).\\
For (b), we have
\[B_{\infty}(\mathbb{U})\xrightarrow[]{C_{\psi}}B_{\infty}(\mathbb{D})\xrightarrow[]{T}L^{\infty}(\mathbb{D})\xrightarrow[]{C_{\psi ^{-1}}}L^{\infty}(\mathbb{U}).\]  Now, $C_{\psi}$ is an isometry of $B_{\infty}(\mathbb{U})$ onto $B_{\infty}(\mathbb{D})$ up to constant, $T$ is a bounded embedding of $B_{\infty}(\mathbb{D})$ into $L^{\infty}(\mathbb{D})$ \cite[Lemma 5.14]{K.ZH}, $C_{\psi^{-1}}$ is also an isometry of $L^{\infty}(\mathbb{D})$ onto $L^{\infty}(\mathbb{U})$, it therefore follows that $S=C_{\psi ^{-1}} T C_{\psi}$ is a bounded embedding of $B_{\infty}(\mathbb{U})$ into $L^{\infty}(\mathbb{U})$.\\
For (c), we have
\[B_{\infty,\circ}(\mathbb{U})\xrightarrow[]{C_{\psi}}B_{\infty,\circ}(\mathbb{D})\xrightarrow[]{T}\mathbb{C_{\circ}}(\mathbb{D})
\xrightarrow[]{C_{\psi^{-1}}}\mathbb{C_{\circ}}(\mathbb{U}).\]
$C_{\psi}$ is a bijection of $B_{\infty,\circ}(\mathbb{U})$ into $B_{\infty,\circ}(\mathbb{D})$, $T$ is an embedding of $B_{\infty,\circ}(\mathbb{D})$ into $\mathbb{C_{\circ}}(\mathbb{D})$ \cite[Lemma 5.14]{K.ZH}, and on the other hand, $C_{\psi^{-1}}$ is also a bijection of $\mathbb{C_{\circ}}(\mathbb{D})$ into $\mathbb{C_{\circ}}(\mathbb{U})$. Therefore $S=C_{\psi ^{-1}} T C_{\psi}$ is an embedding of $B_{\infty,\circ}(\mathbb{U})$ into $\mathbb{C_{\circ}}(\mathbb{U})$, which completes the proof.
\end{proof}

We now establish the predual space of $L^{1}_{a}(\mathbb{U},\mu_{\alpha})$ as we give in the following theorem:
\begin{thm}\label{Theorem 9}
For any $\alpha>-1$, we have;
\begin{eqnarray*}
  (B_{\infty,\circ}(\mathbb{U},i))^{*} &\approx& L^{1}_{a}(\mathbb{U},\mu_{\alpha}),
\end{eqnarray*}
under the pairing
\begin{eqnarray*}
  \langle g,f\rangle &=& \int_{\mathbb{U}}g(\w)\overline{f(\w)}d\mu_{\alpha}(\w),
\end{eqnarray*}
where $g \in B_{\infty,\circ}(\mathbb{U},i)$ and $f \in L^{1}_{a}(\mathbb{U},\mu_{\alpha})$. Here, $B_{\infty,\circ}(\mathbb{U},i)$ is equipped with the same norm as $B_{\infty}(\mathbb{U},i)$.
\end{thm}
\begin{proof}
If $f\in L^{1}_{a}(\mathbb{U},\mu_{\alpha})$, then by Theorem \ref{Theorem 7} above,
$g\longmapsto  \int_{\mathbb{U}}g(\w)\overline{f(\w)}d\mu_{\alpha}(\w)$
defines a bounded linear functional on $B_{\infty,\circ}(\mathbb{U},i)$. Conversely, if $F$ is a bounded linear functional on $B_{\infty,\circ}(\mathbb{U},i)$, we want to show that there exists a function $f\in L^{1}_{a}(\mathbb{U},\mu_{\alpha})$ such that $F(g) = \int_{\mathbb{U}}g(\w)\overline{f(\w)}d\mu_{\alpha}(\w)$ for $g$ in a dense set of $B_{\infty,\circ}(\mathbb{U},i)$.\\
Now we fix any positive parameter $t$ and consider the embedding $S$ of $B_{\infty,\circ}(\mathbb{U},i)$ into $\mathbb{C}_{\circ}(\mathbb{U})$ as given by Prop \ref{prop 2}.
The space $X=S(B_{\infty,\circ}(\mathbb{U},i))$ is a closed subspace of $\mathbb{C}_{\circ}(\mathbb{U})$ and $F\circ S^{-1}:X\rightarrow \mathbb{C}$ is a bounded linear functional on $X$ since $F$ and $S^{-1}$ are both bounded.
By the Hahn-Banach extension theorem, $F\circ S^{-1}$ extends to a bounded linear functional on $\mathbb{C}_{\circ}(\mathbb{U})$. By the Riesz representation theorem, there exists a finite weighted measure $\mu_{\alpha}$ on $\mathbb{U}$ such that $\|\mu_{\alpha}\|=\|F\circ S^{-1}\|$ and $F\circ S^{-1}(h)=\int_{\mathbb{U}}h(z)d\mu_{\alpha}(z)$, $h\in\mathbb{C}_{\circ}(\mathbb{U})$.
In particular, if $g$ is a polynomial (polynomials are dense in $B_{\infty,\circ}(\mathbb{U},i)$), then $F(g)=F\circ S^{-1}\circ S(g)=\int_{\mathbb{U}}Sg(z)d\mu_{\alpha}(z)$.
By Fubini's theorem, we have $F(g)=\int_{\mathbb{U}}g(\w)\overline{f(\w)}d\mu_{\alpha}(\w)$, where $f=C_{\psi ^{-1}} T C_{\psi}$ which is bounded since $T$ is bounded.
\end{proof}

\section{Groups of weighted composition operators on predual of $L^{1}_{a}(\mathbb{U},\mu_{\alpha})$}
As remarked in the section 1, the automorphisms of the upper half plane $\uP$ was classified into three distinct groups in \cite{BBMM}, namely: the scaling, the translation and the rotation groups. Since the corresponding groups of composition operators for the rotation group are defined on the analytic spaces of the unit disk, we shall only consider groups of composition operators associated with the scaling and the translation groups in this paper. It will turn out that these are strongly continuous groups of invertible isometries on the Bloch space $B_{\infty,\circ}(\uP,i)$. We shall identify the infinitesimal generator of each group and determine the spectra of both the generator as well as the resulting resolvents. These results complete the analysis of the adjoints of the weighted composition groups on the weighted Bergman spaces of the upper half plane initiated by \cite{BBMM} and \cite{Bon}. 
\subsection{Scaling group}
The automorphisms of this group are of the form $\vj_t(z)=k^tz$, where $z\in \uP$ and $k, t \in \R$ with $k\neq 0$. As noted in \cite{BBMM} and without loss of generality, we consider the analytic self maps $\varphi_{t}:\,\mathbb{U}\,\to\,\mathbb{U}$ of the form $\varphi_{t}(z)= e^{-t}z$ for $z\in \mathbb{U}$. The corresponding group of weighted composition operators on $L^{p}_{a}(\mathbb{U},\mu_{\alpha})$ is given by $T_{t}f(z) = e^{-t\gamma}f(e^{-t}z)$,
for all $f\in L^{p}_{a}(\mathbb{U},\mu_{\alpha})$, where $\gamma$=$\frac{\alpha+2}{p}$ and $1\leq p < \infty$.
For $p=1$, $(T_{t})_{t \geq 0}$ is defined on $L^{1}_{a}(\mathbb{U},\mu_{\alpha})$ with $\g=\al+2$.\\
Following Theorem \ref{Theorem 9}, the predual of $L^{1}_{a}(\mathbb{U},\mu_{\alpha})$ is given by the duality relation 
\begin{equation}\label{eq3.1}
(B_{\infty,\circ}(\mathbb{U},i))^{*}\approx  L^{1}_{a}(\mathbb{U},\mu_{\alpha}) 
\end{equation}
under the integral pairing
\begin{equation}\label{eq3.2}
\langle g,f \rangle = \int_{\mathbb{U}}g(w)\overline{f(w)}d\mu_{\alpha}(w),
\end{equation}
where $g\in B_{\infty,\circ}(\mathbb{U},i)$ and  $f \in L^{1}_{a}(\mathbb{U},\mu_{\alpha})$.\\
Using the duality pairing above, we obtain the corresponding group of weighted composition operators on $B_{\infty,\circ}(\mathbb{U},i)$ as below:\\
Let $g\in B_{\infty,\circ}(\mathbb{U},i)$ and  $f \in L^{1}_{a}(\mathbb{U},\mu_{\alpha})$, then,
\begin{eqnarray*}
  \langle g,T_{t}f \rangle &= &\int_{\mathbb{U}}g(z)\overline{e^{-t\gamma}f(e^{-t}z)}d\mu_{\alpha}(z)  \\
   &= &\int_{\mathbb{U}}g(z)e^{-t\gamma}\overline{f(e^{-t}z)}(\Im(z))^{\alpha}dA(z).
\end{eqnarray*}
By change of variables, let $\w=e^{-t}z$, then $z=e^{t}\w$,
$dA(\w)=e^{-2t}dA(z)$ and $\Im(z)=e^{t}Im(\w)$.
 Then, \begin{eqnarray*}
   \langle g,T_{t}f \rangle &=&\int_{\mathbb{U}}g(e^{t}\w)e^{-t\gamma}\overline{f(\w)}e^{\alpha t}(\Im(\w))^{\alpha}e^{2t}dA(\w)\\
        &=&\int_{\mathbb{U}}g(e^{t}\w)e^{-t\gamma}e^{t\gamma}\overline{f(\w)}d\mu_{\alpha}(\w)  \\
        &=&\int_{\mathbb{U}}g(e^{t}\w)\overline{f(\w)}d\mu_{\alpha}(\w)=\langle T^*_{t}g,f \rangle.
 \end{eqnarray*}
Now, we define $S_t:=T_t^*$ on $B_{\infty,\circ}(\mathbb{U},i)$ and therefore $S_{t}g(\w) :=  g(e^{t}\w)$ is a semigroup or group of composition operators defined on $B_{\infty,\circ}(\mathbb{U},i)$. We shall carry out a complete study of both the semigroup and spectral properties of this group. We begin by proving the strong continuity property.

%Next, now we study the semigroup properties of $(S_{t})_{t\geq 0}$ on $B_{\infty,\circ}(\mathbb{U},i)$.\\
%For the semigroup properties, we investigate the semigroup properties and determine the infinitesimal generator $\Gamma$ of $(S_{t})_{t\geq 0}$ on $B_{\infty,\circ}(\mathbb{U},i)$ where $S_{t}g(w) =  g(e^{t}w)$ using the following theorems.

\begin{thm}
Let $S_{t}g(w) :=  g(e^{t}w)$ be a semigroup of composition operators defined on $B_{\infty,\circ}(\mathbb{U},i)$. Then, $ (S_{t})_{t\in \mathbb{R}}$ is a strongly continuous group of isometries on $B_{\infty,\circ}(\mathbb{U},i).$
\end{thm}
\begin{proof}
It is clear from the definition that $ (S_{t})_{t\in \mathbb{R}}$ is a group. To prove that $ (S_{t})_{t\in \mathbb{R}}$ is an isometry on $B_{\infty,\circ}(\uP,i)$, we have;
\begin{eqnarray*}
  \|S_{t}g\|_{B_{\infty,\circ}(\mathbb{U},i)} &=& \sup_{\w\in\mathbb{U}}\Im(\w)|S_{t}g^{\prime}(\w)| \\
   &=& \sup_{\w\in\mathbb{U}}\Im(\w)e^{t}|g^{\prime}(e^{t}\w)|.
\end{eqnarray*}
Now by change of variables, let $z=e^{t}\w$ then $\w=e^{-t}z$, and $\Im(\w)=e^{-t}\Im(z)$. Therefore,
\begin{eqnarray*}
  \|S_{t}g\|_{B_{\infty,\circ}(\mathbb{U},i)}&=&\sup_{z\in\mathbb{U}}e^{-t}\Im(z)e^{t}|g^{\prime}(z)|\\
   &=&\sup_{z\in\mathbb{U}}\Im(z)|g^{\prime}(z)|  \\
   &=& \|g\|_{B_{\infty,\circ}(\mathbb{U},i)},\mbox{as desired}.
\end{eqnarray*}
For strongly continuity, we first take note that $S_{t}=C_{\varphi_{-t}}$ since $S_tg(\w)=g(\vj_t(\w))$. Then by Proposition \ref{prop1.2}, it is easy to see that $C_{\psi_{-t}}$ is strongly continuous on $B_{\infty,\circ}(\mathbb{U},i)$ if and only if $(C_{\psi^{-1}\circ\varphi_{-t}\circ\psi})_{t\in\mathbb{R}}$ is strongly continuous on $B_{\infty,\circ}(\mathbb{D},0)$, where $B_{\infty,\circ}(\D,0)$ is the subspace of $B_{\infty,\circ}(\mathbb{D})$ consisting of functions vanishing at point $0$.
Now by simple computation of $\psi^{-1}\circ\varphi_{-t}\circ\psi(z)$, we obtain;
\begin{eqnarray*}
   \psi^{-1}\circ\varphi_{-t}\circ\psi(z)&=&\frac{z-\frac{1-e^{t}}{1+e^{t}}}{1-\frac{1-e^{t}}{1+e^{t}}z}\\
   &=& \frac{z-a_{t}}{1-\overline{a}_{t}z},
\end{eqnarray*}
where $a_{t}=\frac{1-e^{t}}{1+e^{t}}$. As $t\to 0$, $a_t \to 0$. Let $h_{a}(z)= \frac{z-a_{t}}{1-\overline{a}_{t}z}=\psi^{-1}\circ\varphi_{-t}\circ\psi(z)$, then for strong continuity, it therefore suffices to show that $\|C_{h_{a}}f-f\|_{B_{\infty,\circ}(\mathbb{D},0)}\,\to\,0$ as $a\rightarrow0$ ($a_t\to 0$).
Using the density of polynomials in ${B_{\infty,\circ}(\mathbb{D},0)}$, let $f(z)=z^{n}$. Then $C_{h_{a}}z^{n}-z^{n} = (h_{a}(z))^{n}-z^{n}, n\geq1$, and
\begin{eqnarray*}
(C_{h_{a}}f-f)^{\prime}(z)&=& n[(h_{a}(z))^{n-1}h^{\prime}_{a}(z)-z^{n-1}].
\end{eqnarray*}
But $h_{a}(z)=\frac{z-a_{t}}{1-\overline{a}_{t}z}$, and hence $h^{\prime}_{a}(z) = \frac{1-a_{t}\overline{a}_{t}}{(1-\overline{a}_{t}z)^{2}}.$
Therefore,
\begin{eqnarray*}
   (C_{h_{a}}f-f)^{\prime}(z)&=&n\left[\frac{(h_{a}(z))^{n-1}(1-a_{t}\overline{a}_{t})}{(1-\overline{a}_{t}z)^{2}}-z^{n-1}\right]  \\
   &=&n\left[\frac{(\frac{z-a_{t}}{1-\overline{a}_{t}z})^{n-1}(1-a_{t}\overline{a}_{t})}{(1-\overline{a}_{t}z)^{2}}-z^{n-1}\right]  \\
   &=&n\left[\frac{(z-a_{t})^{n-1}(1-a_{t}\overline{a}_{t})-z^{n-1}((1-\overline{a}_{t}z)^{n+1})}{(1-\overline{a}_{t}z)^{n+1}}\right].
\end{eqnarray*}
Now,
\begin{eqnarray*}
    \lim_{a \to 0}\|C_{h_{a}}f-f\|_{B_{\infty,\circ(\mathbb{D},0)}}&=&\lim_{a\to 0}\left(\sup_{z\in\mathbb{D}}(1-|z|^{2})|(C_{h_{a}}f-f)^{\prime}|(z)\right)  \\
       &=&\lim_{t\to 0}\left(\sup_{z\in\mathbb{D}}(1-|z|^{2})\left|n\left[\frac{(z^{n-1})(1)-z^{n-1}(1)}{(1)^{n+1}}\right]\right|\right)  \\
       &=&\lim_{t\to 0}\left(\sup_{z\in\mathbb{D}}(1-|z|^{2})\left|n[z^{n-1}-z^{n-1}]\right|\right)  \\
       &=&0.
\end{eqnarray*}
Hence, $(S_{t})_{t\in \mathbb{R}}$ is strongly continuous on $B_{\infty,\circ}(\mathbb{U},i)$, as claimed.
\end{proof}
\begin{thm}
The infinitesimal generator $\Gamma$ of $(S_{t})_{t\geq 0}$ on $B_{\infty,\circ}(\mathbb{U},i)$ is given by $\Gamma g(\w)$=$\w g^{\prime}(\w)$ with the domain $\dom(\Gamma)=\{g\in B_{\infty,\circ}(\mathbb{U},i):\w g^{\prime}(\w)\in B_{\infty,\circ}(\mathbb{U},i)\}.$
\end{thm}
\begin{proof}
By definition, the infinitesimal generator denoted by $\Gamma$ of $(S_{t})_{t\geq 0}$ is given by;
\begin{eqnarray*}
  \Gamma g(\w)=\lim_{t\rightarrow 0^{+}}\frac{g(e^{t}\w)-g(\w)}{t}&= &\left.\frac{\partial}{\partial t}g(e^{t}\w)\right|_{t=0}\\
   &=& \w g^{\prime}(\w).
\end{eqnarray*}
It therefore follows that $\dom(\Gamma)\subseteq \{g\in B_{\infty,\circ}(\mathbb{U},i):\w g^{\prime}(\w)\in B_{\infty,\circ}(\mathbb{U},i)\}.$
To prove the reverse inclusion, we let $g\in B_{\infty,\circ}(\mathbb{U},i)$ be such that $\w g^{\prime}(\w)\in B_{\infty,\circ}(\mathbb{U},i).$
Then for $\w \in \mathbb{U}$, we have;
\begin{eqnarray*}
   S_{t}g(\w)-g(\w)&=&\int _{0}^{t}\frac{\partial}{\partial s}g(e^{s}\w)\,ds \\
   &=&\int _{0}^{t}e^s\w g' (e^{s}\w)\,ds\\
   &=& \int _{0}^{t} S_s G(\w)\,ds\,\,\mbox{ where } G(\w)=\w g'(\w).
\end{eqnarray*}
Thus, \[\lim_{t\to 0^+} \frac{S_t g - g}{t} = \lim_{t\to 0^+}\tfrac1{t}\int_0^t S_s G(\w)\,ds\]
and strong continuity of $(S_s)_{t\geq 0}$ implies that $\frac{1}{t}\int _{0}^{t}\|S_{s}G-G\|ds\,\to\, 0$ as $t\rightarrow 0^{+}$.
Hence $ \dom(\Gamma)\supseteq \{g\in B_{\infty,\circ}(\mathbb{U},i):\w g^{\prime}(\w)\in B_{\infty,\circ}(\mathbb{U},i)\}$, which completes the proof.
\end{proof}
Now for the spectral properties, we obtain the spectra of the generator $\Gamma$, determine the resolvents and further obtain the spectra and the norms of the resulting resolvents.
\begin{thm}\label{THM}
Let $\Gamma$ be the infinitesimal generator of $(S_{t})_{t\in \mathbb{R}}$ on $B_{\infty,\circ}(\mathbb{U},i)$. Then $\sigma_{p}(\Gamma)=\emptyset$ and $\s(\G)=i\R$. In particular, $\Gamma$ is an unbounded operator on $B_{\infty,\circ}(\mathbb{U},i)$.
\end{thm}

Before we prove this theorem, we first give the following Lemma:
%\begin{lem}\label{lem 2}
%For $f\in B_{\infty,\circ}(\mathbb{U})$, we have $ \|f\|_{B_{\infty,\circ}(\mathbb{U})}=\frac{1}{2}\|f\circ\psi\|_{B_{\infty,\circ}(\mathbb{D})}.$
%In particular $f\in B_{\infty,\circ}(\mathbb{U})$ if and only if $f\circ\psi\in B_{\infty,\circ}(\mathbb{D})$.
%\end{lem}
%\begin{proof}
%Let $f$ be a function on $B_{\infty,i}(\mathbb{U})$, then by definition $\lim_{Im(w)\rightarrow0}Im(w)|f^{\prime}(w)|=0$.
%But we know that $Im(w)= \frac{1-|z|^{2}}{|1-z|^{2}}$, hence $ \|f\|_{B_{\infty,i}(\mathbb{U})} =\lim_{|z|\rightarrow1} \left(\frac{1-|z|^{2}}{|1-z|^{2}}\right)|f^{\prime}(\psi(z))|.$
%Now using the chain rule, we obtain
%\begin{eqnarray*}
%% \nonumber to remove numbering (before each equation)
%  \|f\|_{B_{\infty,i}(\mathbb{U})} &=&\frac{1}{2}\lim_{|z|\rightarrow1}(1-|z|^{2})|\psi^{\prime}(z)||f^{\prime}(\psi(z))|.
%   \end{eqnarray*}
%But it is clear that $|\psi^{\prime}(z)||f^{\prime}(\psi(z))|=|(f\circ\psi)^{\prime}(z)|$, and hence
%\begin{eqnarray*}
%% \nonumber to remove numbering (before each equation)
%  \|f\|_{B_{\infty,i}(\mathbb{U})} &=&\frac{1}{2}\lim_{|z|\rightarrow1}(1-|z|^{2})|(f\circ\psi)^{\prime}(z)|.
%   \end{eqnarray*}
%Therefore, $\|f\|_{B_{\infty,i}(\mathbb{U})} = \frac{1}{2}\|f\circ\psi\|_{B_{\infty,\circ}(\mathbb{D})}.$
%\end{proof}
\begin{lem}\label{lem3}
If $\nu \in \C$ and $c\in \R$, we have
\begin{enumerate}
  \item $g(\w) = c \w^{\nu} \notin B_{\infty,0}(\uP)$ for any $c$
  \item $f(\w)=(w-i)^{\nu} \in B_{\infty,0}(\uP)$ if and only if $\Re(\nu)< 0.$\\
  In particular, $g(\w) \notin B_{\infty,0}(\uP,i)$ for any $c$ and $f(\w) \in B_{\infty,0}(\uP,i)$ if and only if $\Re(\nu)< 0.$
\end{enumerate}
\end{lem}
\begin{proof}
From Proposition \ref{prop1.2}, we know that $g\in B_{\infty,\circ}(\mathbb{U})$ if and only if $g \circ \psi \in B_{\infty,\circ}(\mathbb{D})$. Then for $z\in \D$,
\begin{align*}
(g \circ \psi)(z)&= g(\psi(z))= c(\psi(z))^{\nu}= c(\frac{i(1+z)}{1-z})^{\nu}\\
&= c i(1+z)^{\nu}(1-z)^{-\nu}.
\end{align*}
Now $g\circ \psi \in \Hol(\D)$ if and only if $\Re(\nu) >0$ and $\Re(-\nu)>0$ which is not possible, and therefore $g\circ \psi \notin \Hol(\D)$. Hence $g\notin B_{\infty,0}(\uP)$. This proves (1).\\
%implying that,
%\begin{eqnarray*}
%         % \nonumber to remove numbering (before each equation)
%           (g \circ \psi )^{\prime}(z) &=& c\lambda\left(\frac{i(1+z)}{1-z}\right)^{\lambda-1}.\frac{i(1-z)+i(1+z)}{(1-z)^{2}}  \\
%            &=& 2ic\lambda (i(1+z))^{\lambda-1}(1-z)^{-(\lambda+1)}.
%         \end{eqnarray*}
%Now for $(1-z)^{-(\lambda+1)}$, $Re(-(\lambda + 1))=Re(-\lambda - 1)>0 \Leftrightarrow Re(\lambda) <-1. $
%And for $(1+z)^{\lambda-1}$, $Re(\lambda-1)=Re(\lambda-1)>0 \Leftrightarrow Re(\lambda)>1. $
%Then $(1+z)^{\lambda-1}\in B_{\infty,\circ}(\mathbb{D})$ if and only if $Re(\lambda)>1 $ and $(1-z)^{-(\lambda+1)}\in B_{\infty,\circ}(\mathbb{D})$ if and only if  $ Re(\lambda) <-1 $. Hence $ Re(\lambda) =0 $. This proves (1).\\

For (2), following \cite[Lemma 3.2]{BBMM}, for any $\nu\in \C$, $(w-i)^{\nu} \in \Hol(\uP)$ if and only if $\Re(\nu)< 0$ since $\g = 0$ in this case.\\
The particular cases follow immediately since $B_{\infty,0}(\uP,i) \subseteq B_{\infty,0}(\uP)$ and $g(i)\neq 0$ for (1), while $f(i)=0$ for (2).
\end{proof}
\begin{proof}[Proof of Theorem \ref{THM}]
To obtain the point spectrum of $\Gamma$, let $\lambda$ be an eigenvalue of $\Gamma$ and $g$ be the corresponding eigenvector. Then
 $\Gamma g(\w)=\lambda g(\w)$ is equivalent to $\w g^{\prime}(\w)= \lambda g(\w)$ which yields $\frac{\w g^{\prime}(\w)}{\w} = \frac{ \lambda g(\w)}{\w}$ by dividing both sides by $\w$. By integrating both sides, we obtain $g(\w)=c \w^{\lambda}$, which is not in $ B_{\infty,\circ}(\mathbb{U},i)$ for any $c$. Therefore $\sigma_{p}(\Gamma) = \emptyset$.\\
Since each $S_{t}$ is an invertible isometry, its spectrum satisfies $\sigma (S_{t})\subseteq\partial \mathbb{D}$. Therefore the spectral mapping theorem for strongly continuous groups \cite[Theorem 2.3]{Paz} implies that $e^{t\sigma(\Gamma)}\subseteq \sigma (S_{t})\subseteq \partial \mathbb{D}$. Now let $\lambda\in \sigma (\Gamma)$, then
$|e^{t\lambda}| = 1$ which further implies that $\Re(\lambda)=0$. Thus $\lambda\in i\mathbb{R}$ and therefore $\sigma(\Gamma)\subseteq i\mathbb{R}$.\\
We now need to show the reverse inclusion, that is, $i\R \subseteq \s(\G)$. Fix $\lambda \in i\R$ and assume $\la \notin \s(\G)$ which implies that the resolvent operator $R(\la,\G):\,B_{\infty,\circ}(\mathbb{U},i)\,\to\,B_{\infty,\circ}(\mathbb{U},i)$ is bounded. Consider the function $h(w)=(w-i)^{-(\la +1)}$. Then $\Re(-(\la +1)) = -1 < 0$ and following Lemma \ref{lem3}, it is immediate that $h\in B_{\infty,\circ}(\mathbb{U},i)$. The image function $f=R(\la,\G)h$ is equivalent to $(\la-\G)f=h$ which yields a differential equation
\[f'(\w) - \frac{\la}{\w}f(\w) = -\frac{h(\w)}{\w},\]
whose general solution is
\[f(\w)= (\w-i)^{-\la}+c \w^{\la} \]
which does not belong to $B_{\infty,\circ}(\mathbb{U},i)$ for any $c$, by Lemma \ref{lem3}. Thus $h\notin \ran (\la-\G)$ and so $\s(\G)=i\R.$
\end{proof}
\begin{thm}
Let $\Gamma$ be the infinitesimal generator of $(S_{t})_{t\in \mathbb{R}}$. Then the following hold;
\begin{enumerate}
  \item For $\lambda \in \rho(\Gamma)$, and $h\in B_{\infty,\circ}(\mathbb{U},i)$ then,
  \begin{itemize}
    \item [(i)] $R(\lambda,\Gamma)h(\w)= \w^{\lambda}\int^{\infty}_{\w} \frac{1}{z^{\lambda+1}}h(z)\, dz$, if $\Re(\lambda)>0$.
    \item [(ii)]$R(\lambda,\Gamma)h(\w)=-\w^{\lambda}\int^{\w}_{0} \frac{1}{z^{\lambda+1}}h(z) \,dz,$ if $\Re(\lambda)<0$.
  \end{itemize}
\item  $\sigma(R(\lambda,\Gamma))=\left\{\w: |\w-\frac{1}{2\Re(\lambda)}|=\frac{1}{2\Re(\lambda)}\right\}. $
\item $r(R(\lambda,\Gamma))=\|R(\lambda,\Gamma)\|= \frac{1}{|\Re(\lambda)|}$.
\end{enumerate}
\end{thm}
\begin{proof}
To prove (1), we take note the resolvent set is given as $\rho(\Gamma)=\{ \lambda \in \mathbb{C}:Re(\lambda)\neq0$\}. We therefore consider the following cases:\\
\textbf{Case 1}: If $Re(\lambda)>0$, then the resolvent operator is given by the Laplace transform: For every $h\in B_{\infty,\circ}(\mathbb{U},i)$, we have $R(\lambda,\Gamma)h=\int^{\infty}_{0}e^{-\lambda t}S_{t}hdt$ with convergence in norm. Therefore, $R(\lambda,\Gamma)h=\int^{\infty}_{0}e^{-\lambda t}h(e^{t}\w)dt.$ By change of variables, let $z=e^{t}\w$, then $\w=e^{-t}z$, $\frac{dz}{dt}=\w e^{t}$ then $dt=\frac{dz}{\w e^{t}}=\frac{dz}{z}$.
Therefore when $t= 0 \Rightarrow z= \w$ and  $t= \infty \Rightarrow z=\infty$, and so;
\begin{eqnarray*}
   R(\lambda,\Gamma)h(\w)&=&\int^{\infty}_{\w}e^{-\lambda t} h(z)\frac{dz}{z} = \int^{\infty}_{\w}\left(\frac{z}{\w}\right)^{-\lambda} \frac{1}{z}h(z)dz\\
   &=&  \w^{\lambda}\int^{\infty}_{\w} \frac{1}{z^{\lambda+1}}h(z)dz.
\end{eqnarray*}
\textbf{Case 2}: If $Re(\lambda)<0$, then $R(\lambda,\Gamma)h = -R(-\lambda,-\Gamma)h= -\int^{\infty}_{0}e^{\lambda t} h(e^{-t} \w)dt.$
Then again by change of variables, let $z=e^{-t}\w$, then $e^{t}=\frac{\w}{z}$, $\frac{dz}{dt}=-\w e^{-t}$ and $dt=\frac{-dz}{\w e^{-t}}=-\frac{dz}{z}$.
Therefore $t=0 \Rightarrow z=w$ and $t=\infty\Rightarrow z=0$ and so;
\begin{eqnarray*}
  R(\lambda,\Gamma)h(w) &=& -\int^{0}_{\w}e^{\lambda t} h(z).-\frac{dz}{z}=-\int^{\w}_{0}\left(\frac{\w}{z}\right)^{\lambda}h(z).\frac{dz}{z}  \\
   &=& -\w^{\lambda}\int^{\w}_{0} \left(\frac{1}{z}\right)^{\lambda}.\frac{1}{z}h(z)dz  \\
   &=& -\w^{\lambda}\int^{\w}_{0} \frac{1}{z^{\lambda+1}}h(z)dz.
\end{eqnarray*}
To prove (2), we use the spectral mapping theorem for the resolvents which asserts that $\sigma(R(\lambda,\Gamma))=\left\{\frac{1}{\lambda-\mu}:\mu\in\sigma(\Gamma)\right\}\setminus \{0\}$ for $\lambda\in\rho(\Gamma)$. Therefore,
\begin{eqnarray*}
   \sigma(R(\lambda,\Gamma))  &=&\left\{\frac{1}{\lambda-ir}:r\in\mathbb{R}\right\}\setminus \{0\}\\ &=&\left\{\frac{1}{\Re(\lambda)+i(Im(\lambda)-r)}:r\in\mathbb{R} \right\}\setminus \{0\}.
  \end{eqnarray*}
Rationalizing the denominator and simplifying we get \\ $\sigma(R(\lambda,\Gamma))=\left\{\frac{(\Re(\lambda)-i(\Im(\lambda)-r))}{(\Re(\lambda))^{2}+(\Im(\lambda)-r)^{2}}:r\in\mathbb{R}\right\}.$\\
Now by letting $w = \frac{(\Re(\lambda)-i(\Im(\lambda)-r))}{(\Re(\lambda))^{2}+(\Im(\lambda)-r)^{2}}$, subtracting $\frac{1}{2\Re(\lambda)}$ and finding the magnitude of both sides we get,
\begin{eqnarray*}
\left|w-\frac{1}{2\Re(\lambda)}\right|^{2} &=& \frac{1}{(2\Re(\lambda))^{2}},
\end{eqnarray*}
and so
\begin{eqnarray*}
\left|w-\frac{1}{2\Re(\lambda)}\right|&=&  \frac{1}{2\Re(\lambda)}.
\end{eqnarray*}
Therefore, $\sigma(R(\lambda,\Gamma))=\left\{w:|w-\frac{1}{2\Re(\lambda)}|=\frac{1}{2\Re(\lambda)}\right\}.$
%Similarly, the point spectrum of the resolvent operator is given by;
%\begin{eqnarray*}
%% \nonumber to remove numbering (before each equation)
%   \sigma_{p}(R(\lambda,\Gamma))\setminus \{0\} &=&(\lambda-\sigma_{p}(\Gamma))^{-1}=\left\{\frac{1}{\lambda-\mu}:\mu\in\sigma_{p}(\Gamma)\right\}.
%\end{eqnarray*}
%Since $\sigma(\Gamma)=\sigma_{p}(\Gamma)$ as proved in (1), it therefore follows that
%\begin{eqnarray*}
%% \nonumber to remove numbering (before each equation)
%  \sigma_{p}(R(\lambda,\Gamma)) &=& \sigma(R(\lambda,\Gamma))= \left\{w:\left|w-\frac{1}{2Re(\lambda)}\right|=\frac{1}{2Re(\lambda)}\right\}\\
%   &\Rightarrow&\sigma_{p}(R(\lambda,\Gamma))= \left\{w:\left|w-\frac{1}{2Re(\lambda)}\right|=\frac{1}{2Re(\lambda)}\right\}.
%\end{eqnarray*}
For part (3), the spectral radius $r(R(\lambda,\Gamma))$ is given by;
\begin{eqnarray*}
   r(R(\lambda,\Gamma))&=& \sup\{|w|:w \in\sigma(R(\lambda,\Gamma))\} \\
   &=& \sup\left\{ |w|:\left|w-\frac{1}{2Re(\lambda)}\right|=\frac{1}{2Re(\lambda)}\right\}= \frac{1}{|Re(\lambda)|}.
\end{eqnarray*}
Finally to determine $\|R(\lambda,\Gamma)\|$, we use the Hille Yosida theorem as well as the fact that the spectral radius is always bounded by the norm.
%Since $r(R(\lambda,\Gamma))=\frac{1}{|Re(\lambda)|}$, then using the Hille Yosida theorem which asserts that
%for every $\lambda \geq 0$, $\| R(\lambda , \Gamma) \|\leq \frac{1}{|Re(\lambda)|}.$
Therefore,
\begin{eqnarray*}
 \frac{1}{|Re(\lambda)|}=r(R(\lambda,\Gamma))&\leq&\|R(\lambda,\Gamma)\|\leq \frac{1}{|Re(\lambda)|}.
% &\Rightarrow& \|R(\lambda,\Gamma)\|= \frac{1}{|Re(\lambda)|}.
\end{eqnarray*}
Thus, $r(R(\lambda,\Gamma))=\|R(\lambda,\Gamma)\|= \frac{1}{|\Re(\lambda)|},$ as desired.
\end{proof}

\subsection{Translation Group}
In this group the automorphisms are of the form $\vj_t(z)=z+kt$, where $z\in \uP$ and $k, t \in \R$ with $k\neq 0.$ As noted earlier in subsection 3.1, without loss of generality we let $k=1$ and consider the self analytic maps $\vj_t:\,\uP\,\to\,\uP$ given by $\vj_t(z)=z+t$ for $z\in \uP.$
Then the corresponding group of weighted composition operators defined on $L^{1}_{a}(\mathbb{U},\mu_{\alpha})$ is therefore given by $T_{t}f(z)= f(z+t)$, for all $f\in L^{p}_{a}(\mathbb{U},\mu_{\alpha})$. \\
Now using the duality relation given by equation \eqref{eq3.1} and its sesquilinear pairing given by equation \eqref{eq3.2}, we have:

Let $g\in B_{\infty,\circ}(\mathbb{U},i)$ and $f\in L_a^1(\uP,\mu_\al)$, then 
\begin{align*}
\langle g,T_{t}f \rangle &=\int_{\mathbb{U}}g(z)\overline{f(z+t)}d\mu_{\alpha}(z)\\
&=\int_{\mathbb{U}}g(z)\overline{f(z+t)}(\Im(z))^{\alpha}dA(z).
\end{align*}
Now by a change of variables, let $\w = z+t $, then $ z = \w-t$ and $dA(\w) = dA(z)$.
Therefore,
\begin{eqnarray*}
  \langle g,T_{t}f \rangle &=& \int_{\mathbb{U}}g(\w-t)\overline{f(\w)}(\Im(\w))^{\alpha}dA(\w) \\
   &=& \int_{\mathbb{U}}g(\w-t)\overline{f(\w)}d\mu _{\alpha}(\w)\\
   &=& \langle T^*_{t}g, f\rangle.
\end{eqnarray*}
Now, define $S_t:=T_t^*$ on $B_{\infty,0}(\uP,i)$. Then for $g\in B_{\infty,0}(\uP,i)$, $S_{t}g(\w)= g(\w-t)$ is the group of composition operators defined on $B_{\infty,\circ}(\mathbb{U},i)$. We determine the semigroup properties of this group in the remaining part of this paper.
\begin{thm}
Let $S_{t}g(\w) := g(\w-t)$ be a semigroup of composition operators defined on $B_{\infty,\circ}(\mathbb{U},i)$. Then, $ (S_{t})_{t\in \mathbb{R}}$ is a strongly continuous group of isometries on $B_{\infty,\circ}(\mathbb{U},i).$
\end{thm}
\begin{proof}
It is clear from the definition that $ (S_{t})_{t\in \mathbb{R}}$ is a group. To prove that $ (S_{t})_{t\in \mathbb{R}}$ is an isometry, then by the definition of isometry, we have;
\begin{eqnarray*}
  \|S_{t}g\|_{B_{\infty,\circ}(\mathbb{U},i)} &=& \sup_{\w\in\mathbb{U}}\Im(\w)|(S_{t}g)^{\prime}(\w)| \\
  % &=&\sup_{w\in\mathbb{U}}Im(w)|(g(w-t))^{\prime}|  \\
   &=&\sup_{\w\in\mathbb{U}}\Im(\w)|g^{\prime}(\w-t)|.
\end{eqnarray*}
By change of variables, let $z=\w-t$ then $\w=z+t$ and $\Im(\w)=\Im(z)$.
Hence,
\begin{eqnarray*}
  \|S_{t}g\|_{B_{\infty,\circ}(\mathbb{U},i)}&=&\sup_{z\in\mathbb{U}}\Im(z)|g^{\prime}(z)|\\
   &=& \|g\|_{B_{\infty,\circ}(\mathbb{U},i)},\mbox{ as desired}.
\end{eqnarray*}
For strongly continuity propert, we argue as we did in the previous section. We note that $S_{t}=C_{\varphi_{-t}}$ which is strongly continuous on $B_{\infty,\circ}(\mathbb{U},i)$ if and only if $(C_{\psi^{-1}\circ\varphi_{-t}\circ\psi})_{t\in\mathbb{R}}$ is strongly continuous on $B_{\infty,\circ}(\mathbb{D},0)$, which consists of functions vanishing at point $0$. \\
We compute $\psi^{-1}\circ\varphi_{-t}\circ\psi(z)$. Let $a_{t}=\frac{t}{2i+t}$ and $b_{t}=\frac{2i-t}{2i+t}$, then a straight forward calculation yields
   \begin{eqnarray*}
  \psi^{-1}\circ\varphi_{-t}\circ\psi(z) &=&\frac{z-a_{t}}{b_{t}+a_{t}z}  \\
   &=&h_{a}(z),
\end{eqnarray*}
where we have let $h_{a}(z) = \frac{z-a_{t}}{b_{t}+a_{t}z}$. Clearly, $t\,\to\,0$ as $a_{t}\,\to\,0$ and $b_{t}\,\to\,1$.
It therefore suffices to show that $\|C_{h_{a}}f-f\|_{B_{\infty,\circ}(\mathbb{D},0)}\rightarrow0$ as $t\rightarrow0$.
Using density of polynomials in ${B_{\infty,\circ}(\mathbb{D},0)}$, we let $f(z)=z^{n}$. Then $C_{h_{a}}z^{n}-z^{n}=(h_{a}(z))^{n}-z^{n}, n\geq1.$
Therefore $(C_{h_{a}}f-f)^{\prime}(z)=n[(h_{a}(z))^{n-1}h^{\prime}_{a}(z)-z^{n-1}]$.
But $h_{a}(z)=\frac{z-a_{t}}{b_{t}+a_{t}z}\Rightarrow h^{\prime}_{a}(z)=\frac{(b_{t}+a_{t}z)(1)-(z-a_{t})(a_{t})}{(b_{t}+a_{t}z)^{2}}.$
Therefore by substituting,
\begin{eqnarray*}
   (C_{h_{a}}f-f)^{\prime}(z)&=&n[(h_{a}(z))^{n-1}h^{\prime}_{a}(z)-z^{n-1}]  \\
   &=&n\left[\left(\frac{z-a_{t}}{b_{t}+a_{t}z}\right)^{n-1}\frac{(b_{t}+a_{t}z)-(z-a_{t})(a_{t})}{(b_{t}+a_{t}z)^{2}}-z^{n-1}\right]  \\
   &=&n\left[\frac{(z-a_{t})^{n-1}(b_{t}+a_{t}z)-(z-a_{t})(a_{t})}{(b_{t}+a_{t}z)^{n+1}}-z^{n-1}\right].
\end{eqnarray*}
Now, 
\begin{eqnarray*}
  \lim_{t\rightarrow 0^{+}}\|C_{h_{a}}f-f\|_{B_{\infty,\circ(\mathbb{D},0)}} &=&\lim_{t\rightarrow 0^{+}}\left(\sup_{z\in\mathbb{D}}(1-|z|^{2})|(C_{h_{a}}f-f)^{\prime}|(z)\right)  \\
       &=&\lim_{t\rightarrow 0^{+}}\left(\sup_{z\in\mathbb{D}}(1-|z|^{2})\right.\\
       &&\left.\left|n\left[\frac{(z-a_{t})^{n-1}(b_{t}+a_{t}z)-(z-a_{t})(a_{t})}{(b_{t}+a_{t}z)^{n+1}}-z^{n-1}\right]\right|\right)\\
       &=&\lim_{t\rightarrow 0^{+}}\left(\sup_{z\in\mathbb{D}}(1-|z|^{2})\left|\frac{n[z^{n-1}-0-z^{n-1}]}{1}\right|\right)\\
       &=&0.
    \end{eqnarray*}
Hence $(S_{t})_{\in \mathbb{R}}$ is strongly continuous, as claimed.
\end{proof}
\begin{thm}
The infinitesimal generator $\Gamma$ of $(S_{t})_{t\geq 0}$ on $B_{\infty,\circ}(\mathbb{U},i)$ is given by $\Gamma g(\w)= -g^{\prime}(\w)$ with the domain $\dom(\Gamma)=\{g\in B_{\infty,\circ}(\mathbb{U},i):g^{\prime}\in B_{\infty,\circ}(\mathbb{U},i)\}.$
\end{thm}
\begin{proof}
By definition, the infinitesimal generator $\Gamma$ on $B_{\infty,\circ}(\uP,i)$ is given by;
\begin{eqnarray*}
  \Gamma g(\w)&=&\lim_{t\to 0^{+}}\frac{g(\w-t)-g(\w)}{t}= \left.\frac{\partial}{\partial t}g(\w-t)\right|_{t=0} \\
   &=& -g^{\prime}(\w).
\end{eqnarray*}
Therefore $\dom(\Gamma)\subset \{g\in B_{\infty,\circ}(\mathbb{U},i):g^{\prime}\in B_{\infty,\circ}(\mathbb{U},i)\}.$ Conversely, let $g\in B_{\infty,\circ}(\mathbb{U},i)$ be such that $g^{\prime}\in B_{\infty,\circ}(\mathbb{U},i).$
Thus we have;
\begin{eqnarray*}
  \frac{ S_{t}g-g}{t}&=& \tfrac1{t}\int _{0}^{t}\frac{\partial}{\partial s}S_sg\,ds
\end{eqnarray*}
and for every $\w\in \uP$, $\frac{\partial}{\partial s}S_sg(\w) = -g'(\w-s) = S_sg'(\w)$. Thus,
\begin{align*}
\left\|\frac{S_sg-g}{t} - f' \right\|\leq \tfrac1{t}\int_0^t \left\|T_sf' - f'\right\|\,ds \,\to\,0 \mbox{ as } t\to 0
\end{align*}
by strong continuity, and therefore $ \dom(\Gamma)\supseteq \{g\in B_{\infty,\circ}(\mathbb{U},i):g^{\prime}\in B_{\infty,\circ}(\mathbb{U},i)\}$, which completes the proof.
\end{proof}

%% ------------------------------------------------------------------------
%%
\subsection*{Acknowledgment}
This work was completed during the period when the second author was visiting the Aristotle University of Thessaloniki, Greece. He would like to sincerely that the Simon's Foundation for funding his visit. He would also wish to thank his host Prof. Aristomenis G. Siskakis and the department of Mathematics for the unmatched hospitality

% ------------------------------------------------------------------------
\end{document}